\documentclass[12pt,reqno]{amsart}
\usepackage{amsmath, amssymb, amsfonts, xcolor}
\usepackage{hyperref}
\numberwithin{equation}{section}
\newtheorem{thm}{Theorem}[section]
\newtheorem{example}{Example}
\newtheorem{prop}[thm]{Proposition}
\newtheorem{lem}[thm]{Lemma}
\newtheorem{dfn}[thm]{Definition}
\newtheorem{remark}[thm]{Remark}
\newtheorem{cor}[thm]{Corollary}

\usepackage[a4paper, top = 1.2in, bottom = 1.2in, left = 1.2in, right = 1.2in]{geometry}

\numberwithin{equation}{section}

\newcommand{\F}{\mathbb{F}}
\newcommand{\N}{\mathbb{N}}
\newcommand{\Q}{\mathbb{Q}}

\newcommand{\Z}{\mathbb{Z}}

\newcommand{\mcO}{\mathcal{O}}

\newcommand{\mfb}{\mathfrak{b}}

\newcommand{\mfm}{\mathfrak{m}}
\newcommand{\mfn}{\mathfrak{n}}

\newcommand{\Dif}{\mathfrak{D}}
\newcommand{\mfp}{\mathfrak{p}}
\newcommand{\mfq}{\mathfrak{q}}

\newcommand{\SL}{\mathrm{SL}}

\newcommand{\GL}{\mathrm{GL}}

\newcommand{\Gal}{\mathrm{Gal}}
\newcommand{\Frob}{\mathrm{Frob}}

\def\1{1\!\!1}
\newcommand{\mrm}[1]{\mathrm{#1}}

\title[On generation of the coefficient field by a Fourier coefficient]{On generation of the coefficient field of a primitive Hilbert modular form by a single Fourier coefficient}

\author[N. Kumar]{Narasimha Kumar}
\address[N. Kumar]{Department of Mathematics, Indian Institute of Technology Hyderabad, Kandi, Sangareddy 502285, INDIA.}
\email{narasimha@math.iith.ac.in}

\author[S. Sahoo]{Satyabrat Sahoo}
\address[S. Sahoo]{Department of Mathematics, Indian Institute of Technology Hyderabad, Kandi, Sangareddy 502285, INDIA.}
\email{ma18resch11004@iith.ac.in}

\date{\today}

\begin{document}
	\maketitle
	
	\begin{abstract}
	    For a primitive Hilbert modular form $f$  over $F$ of weight $k$, under certain assumptions on image of $\bar{\rho}_{f,\lambda}$, we calculate the Dirichlet density of primes $\mfp$ for which the $\mfp$-th Fourier coefficient $C(\mfp, f)$ generates the coefficient field $E_f$. 
	    If $k=2$, then we show that the assumption on the image of $\bar{\rho}_{f,\lambda}$ is satisfied when the degrees of $E_f, F$ are equal and odd prime. We also compute the density of primes $\mfp$ for which $C^*(\mfp, f)$ generates $F_f$. Then, we provide some examples of $f$ to support our results.
        Finally, we calculate the density of primes $\mfp$ for which $C(\mfp, f) \in K$ for any field $K$ with $F_f \subseteq K \subseteq E_f$. This density is completely determined by the inner twists of $f$ associated with $K$. This work can be thought of as a generalization of~\cite{KSW08} to primitive Hilbert modular forms.
	\end{abstract}

	\section{Introduction}
	The study of the Fourier coefficients of modular forms is an active area of research in number theory. 
	It is well-known that for any primitive form $f$ over $\Q$, the Fourier coefficients of $f$ generate a number field $E_f$. In~\cite{KSW08}, the authors proved that the set of primes $p$ for which $p$-th Fourier coefficient of $f$ generates $E_f$ has density one if $f$ does not have any non-trivial inner twists. To the best of the author's knowledge, the analogous question is still open for Hilbert modular forms, which is the objective of our study in this article.

	For a primitive form $f$  over $F$, let $E_f$ denote the number field generated by the Fourier coefficients $C(\mfp, f)(\mfp \in P)$ of $f$, where $P$ denote the set of all prime ideals of $F$ (cf.~\cite{S78}). 	We first state a result that, for a primitive form $f$ over $F$ of weight $2$, the set of $\mfp \in P$ for which $\Q(C(\mfp,f))=E_f$ has Dirichlet density $1$, if $[F : \Q]=[E_f : \Q]$ is an odd prime (cf.~Theorem~\ref{MT1}). We then state and prove a general result for $f$ of weight $k$ (cf.~Theorem~\ref{MT2}) under some assumptions on the image of $\bar{\rho}_{f, \lambda}$(cf.~\eqref{assumption in Theorem for E_f})). We then show that these assumptions on the image of $\bar{\rho}_{f, \lambda}$ are satisfied for primitive forms $f$ over $F$ of weight $2$, if $[F : \Q]=[E_f : \Q]$ is an odd prime. The proof of Theorem~\ref{MT1} mainly depends on an important proposition of Dimitrov (cf.~\cite[Proposition 3.9]{D05}).
    We continue this study for the field $F_f \subseteq E_f$ and show that the set of $\mfp \in P$ for which $\Q(C^\ast(\mfp,f))=F_f$ has density $1$ (cf.~\S\ref{Notations section} for the definitions of $C^\ast(\mfp,f)$ and $F_f$).

	This article builds on the ideas of Koo et al., in~\cite{KSW08}, for primitive forms over $\Q$. One of the vital ingredients in the proof of ~\cite[Theorem 1.1]{KSW08} is a Theorem of Ribet (cf.~\cite[Theorem 3.1]{R85}), where he explicitly described the image of $\l$-adic residual Galois representation $\bar{\rho}_{f,l}$ attached to classical modular forms. This result plays a crucial role in obtaining certain sharp bounds for the images of $\bar{\rho}_{f,l}$, which are helpful in their proof. 
	Unfortunately, in our context, an analog of Ribet's result does not seem to exist in the literature. In order to get similar sharp bounds for the images of $\bar{\rho}_{f,\lambda}$, we have to work with some 
	assumptions (cf.~\eqref{assumption in Theorem for E_f} in the text). This explains the reason for our assumptions in Theorem~\ref{MT1} and Theorem~\ref{MT2}. 
% 	However, if $f$ is of parallel weight $2$ with $[F : \Q]=[E_f : \Q]$ is an odd prime, then the assumption~\eqref{assumption in Theorem for E_f} is  satisfied. 
	
	Our results can be thought of as a generalization of the results in~\cite{KSW08} to primitive Hilbert modular forms $f$ over $F$. Using LMFDB, we produce examples of primitive Hilbert modular forms $f$ of parallel weight $2$ in support of Theorem~\ref{MT1} (cf. Example~\ref{example1}, Example~\ref{example2} and Example~\ref{example3} in the text). Finally, we calculate the density of $\mfp \in P$ for which $C(\mfp, f) \in K$, where $K \subseteq E_f$ is a subfield. This density depends on whether $K\supseteq F_f$ or not. If $F_f \not \subseteq K$, then it is zero (cf. Lemma~\ref{K does not contains F_f}), otherwise it is non-zero and completely determined by the inner twists of $f$ associated with $K$ (cf. Proposition~\ref{K contains F_f}).

% 	Further, we shall produce some examples of Hilbert modular forms $f$ for which $f$ satisfies the conditions \eqref{assumption in Theorem for E_f} for the image of $\bar{\rho}_{f, \lambda}$ (\ref{example1}). Here, we make use of the results of Dieulefait and Dimitrov~\cite{DD06}, where they explicitly calculated the images of  $\bar{\rho}_{f, \lambda}$ constructed by Okada and Demb\'el\'e.

 	\subsection{Structure of the article:}
	The article is organized as follows. In $\S2$, we collate all the preliminaries which are required to prove our main theorems (cf.~Theorem~\ref{MT1}, Theorem~\ref{MT2}). We also introduce the notion of inner twists and study their properties quite elaborately. In $\S3$, we state and prove Theorem~\ref{MT1} and its generalization i.e., Theorem~\ref {MT2} for primitive $f$ over $F$ of parallel weight $2$ and 
	weight $k$, respectively, under certain assumptions. We also prove a variant of these results for $F_f$ and study their consequences. In $\S4$, we calculate the Dirichlet  density of $\mfp \in P$ for which $C(\mfp, f) \in K$ for any field $K$ with $K \subseteq E_f$. This density is completely determined by the inner twists of $f$ associated with $K$ if $F_f \subseteq K$.

	\section{Preliminaries}
	\label{Preliminaries}
	Let $F$ be a totally real number field. Let $\mcO_F$, $\mfn$, and $\Dif$ represent the ring of integers, an ideal, and the absolute different of $F$, respectively. 
 
	\subsubsection{Notations}
	\label{Notations section}
	Throughout this article, we fix to use the following notations.
	\begin{itemize}
		\item Let $\mathbb{P}$,  $P$ denote the set of all primes in $\Z$, $\mcO_F$, respectively.
		\item Let $k=(k_1,k_2,...,k_n)$ be an $n$-tuple of integers such that $k_i\geq2$ and $k_1\equiv k_2 \equiv \cdots\equiv k_n \pmod 2$. Let $k_0 := \mrm{max}\{k_1,k_2,...,k_n \}$, $n_0=k_0-2.$ 
		\item For any number field $K$, denote $G_K:= \Gal(\bar{K}/K).$ Let $L$ be a subfield of $K$. 
		For a prime ideal $\mfq$ in $K$ lying above $\mfp=\mfq \cap L$ in $L$, 
        let $\mathrm{e}(\mfq/ \mfp)$ and $\mathrm{f}(\mfq/ \mfp)$ denote the ramification degree and inertia degree of $\mfq$ over $\mfp$, respectively.
		
	\end{itemize}
	
	For any Hecke character $\Psi$ of $F$ with conductor dividing $\mfn$ and infinity-type $2-k_0$, let $S_k(\mfn, \Psi)$ denote the space of all Hilbert modular newforms over $F$ of weight $k$, level $\mfn$ and character $\Psi$. A primitive form is a normalized Hecke eigenform in the space of newforms. The ideal character corresponding to $\Psi$ of $F$ is denoted by $\Psi^\ast.$
	
	For a primitive form $f\in S_k(\mfn, \Psi)$, let 
	$C(\mfb,f)$ denote the Fourier coefficient of $f$ corresponding to an integral ideal  $\mfb$ of $F$ and $C^\ast(\mfb , f):=\frac{{C(\mfb, f)}^2}{\Psi^\ast(\mfb)}$  for all ideal $\mfb$ with $(\mfb, \mfn)=1$. Write  $E_f=\Q(C(\mfb, f)),\ F_f=\Q(C^\ast(\mfb , f))$,  where $\mfb$ runs over all the integral ideals of $\mcO_F$ such that $(\mfb, \mfn)=1$. Let $\mathcal{P}_f$ denote the set of all prime ideals in $E_f$.
	For any two subfields $F_1, F_2$ such that $\Q \subseteq F_2 \subseteq F_1 \subseteq E_f$, we let 
	$$ \mathrm{f}_{\lambda, F_1, F_2} := \mathrm{f}(\lambda \cap F_1 / \lambda \cap F_2)$$
	for $\lambda \in \mathcal{P}_f$.
	The following proposition describes some properties of $E_f$.
	\begin{prop}[\cite{S78}]
		\label{GS}
		Let $f\in S_k(\mfn, \Psi)$ be a primitive form of weight $k$, level $\mfn$ and character $\Psi$ with coefficient field $E_f$. Then 
		\begin{enumerate}
			\item $E_f$ is a finite Galois extension of $\Q$,
			\item $\Psi^\ast(\mfm) \in E_f$, for all ideals $\mfm \subseteq \mcO_F,$
			\item $E_f$ is either a totally real or a CM field,
			\item $E_f = \Q (\{C(\mfp,f)\}_{\mfp \in S})$, where $S \subseteq P$ with $S^c$ is finite,
			\item $\overline{C(\mfp,f)}={\Psi^\ast(\mfp)}^{-1}C(\mfp,f)$,  for all $\mfp \in P$ with $(\mfp,\mfn)=1$.
		\end{enumerate}
	\end{prop}
		
	\subsubsection{Galois representations attached to $f$}
	Let $f\in S_k(\mfn, \Psi)$ be a primitive form of weight $k$, level $\mfn$ and character $\Psi$ with coefficient field $E_f$. For $\lambda \in \mathcal{P}_f$, by the works of Ohta, Carayol, Blasius-Rogawski and Taylor (cf.~\cite{T89} for more details), there exists a continuous Galois representation 
	$$\rho_{f,\lambda} :  G_F \rightarrow \GL_2(E_{f,\lambda}),$$ which is absolutely irreducible, totally odd, unramified outside $\mfn q$, where $q \in \mathbb{P}$ is the rational prime lying below $\lambda$ and $E_{f,\lambda}$ is the completion of $E_f$ at $\lambda$. The representation $\rho_{f,\lambda}$ has the following properties.
	For all primes $\mfp$ of $\mcO_F$ with $(\mfp, \mfn q)=1$, we have 
	\begin{equation}
		\label{trace of Galois representation}
		\mrm{tr}(\rho_{f,\lambda}(\Frob_\mfp))=C(\mfp,f)\ \mathrm{and} \  \det(\rho_{f,\lambda}(\Frob_\mfp))=\Psi^\ast(\mfp) {N(\mfp)}^{k_0-1}
	\end{equation}  
	(cf. ~\cite{C86}). 
	By taking a Galois stable lattice, we define
    \begin{equation}
		\label{residual Galois repres}
		\bar{\rho}_{f, \lambda} := \rho_{f,\lambda} \pmod \lambda: 
		G_F \rightarrow \GL_2(\F_\lambda)
	\end{equation}	
	whose semi-simplification is independent of the particular choice of a lattice.
%     \begin{dfn}[Dirichlet density]
% 		Let $S$ be a non-empty subset of $P$.
% 		The Dirichlet density of $S$ is defined by 
% 		$$\delta_D(S):= \lim\limits_{s\to 1}\frac{\sum_{\mfp \in S} N(\mfp)^{-s}}{\sum_{\mfp \in P} N(\mfp)^{-s}}$$, if the limit exists.
% 		
% 	\end{dfn}
	We conclude this section by recalling the Chebotarev density theorem (cf. \cite{S81}).
	\begin{thm}
		\label{Density of conjugacy classes}
		Let $C$ be a conjugacy class of $G=:\bar{\rho}_{f, \lambda}(G_F)$.
		The natural density of $\{\mfp\in P : [\bar{\rho}_{f, \lambda}(\Frob_\mfp)]_G=C\}$ is $\frac{|C|}{|G|}$.
	\end{thm}
	
	\subsubsection{Inner twists and its properties}
	\label{inner twists}
	In this section, we define inner twists associated with a Hilbert modular form and describe some of their properties. This notion is quite useful in \S\ref{final section}.
	
	Let $f \in S_k(\mfn, \Psi)$ be a primitive form defined over $F$, of weight $k$ with Hecke character $\Psi$. 
	For any Hecke character $\Phi$ of $F$, let $f_{\Phi}$ denote the twist of $f$ by $\Phi$ (cf. \cite[\S 5]{SW93}). The Fourier coefficients of $f$ and $f_{\Phi}$ are related as follows.
	
	\begin{prop}\cite[Proposition 5.1]{SW93}
		\label{TRS}
		Let $f$ and $f_\Phi$ be as above. If $\mfn_0$ and $\mfm_0$ are the conductors of $\Psi$ and $\Phi$, respectively, then 
		$f_\Phi \in  S_k \big(\mrm{lcm}(\mfn,\mfm_0 \mfn_0, \mfm_0^2), \Psi\Phi^2\big)$ and
		$$C(\mfm,f_\Phi)=\Phi^\ast(\mfm)C(\mfm,f),$$ for all ideals $\mfm$ of $\mcO_F$.
	\end{prop}
	\begin{dfn}
		We say a primitive form $f$ is non-CM if there exists a non-trivial Hecke character $\Phi$ of $F$ such that $C(\mfp, f)=\Phi^\ast(\mfp)C(\mfp, f)$ for almost all prime ideals $\mfp$ of $\mcO_F$. 
	\end{dfn}
	We are now ready to define inner twists.
	\begin{dfn}[Inner twists]
		\label{defn inner}
		Let $f\in S_k(\mfn, \Psi)$ be a non-CM primitive form over $F$ of weight $k$, level $\mfn$ and character $\Psi$. For any Hecke character $\Phi$ of $F$, we say the twist $f_\Phi$ of $f$ 
		is inner if there exists a field automorphism $\gamma : E_f \rightarrow E_f$ such that $\gamma(C(\mfp,f))=C(\mfp, f_\Phi)$ for almost all prime ideals $\mfp$ of $\mcO_F$.  
	\end{dfn}

	\begin{remark}
		For a non-CM form $f$, the identity map $\mrm{id} : E_f \rightarrow E_f$ induces an inner twist of $f$ and we refer it as the trivial inner twist of $f$.  
	\end{remark} 
	Let $\Gamma \subseteq \mrm{Aut}(E_f)$ denote the set of all $\gamma$ associated to all the inner twists 
	of $f$. Similar to the classical case, we get that $\Gamma$ is a subgroup of $\mrm{Aut}(E_f)$ and
	$F_f= E_f^{\Gamma}$, where $E_f^{\Gamma}$ is the fixed field of $E_f$ by $\Gamma$. By Galois theory, $E_f$ is a finite Galois extension of $F_f$. Some of the properties of $F_f$ are given below.
	
	\begin{lem}
		\label{main lemma}
		The field $F_f$ is totally real and \,$C^\ast(\mfp,f) \in \Q(C(\mfp,f)).$
	\end{lem}
	\begin{proof}
		By Proposition~\ref{GS}, we have $C^\ast(\mfp,f) =C(\mfp,f)\overline{C(\mfp,f)}$, which shows that $F_f$ is totally real. By Proposition~\ref{GS}, if $E_f$ is totally real, then $C^\ast(\mfp,f)={C(\mfp,f)}^2 \in \Q(C(\mfp,f))$. If $E_f$ is a CM field, then $\Q(C(\mfp,f))$ is preserved under complex conjugation.
		Hence, $C^\ast(\mfp,f) =C(\mfp,f)\overline{C(\mfp,f)} \in \Q(C(\mfp,f))$.
	\end{proof}
	
	We will now examine the existence of trivial and non-trivial inner twists for any primitive form $f$. 
	More precisely,
	\begin{lem}
		\label{existence of non-trivial inner twists}
		If $f \in S_k(\mfn, \Psi) $ is a non-CM primitive form over $F$ with a non-trivial Hecke character $\Psi$, then $f$ has a non-trivial inner twist.
	\end{lem}
	
	\begin{proof}
		Let $\sigma : E_f\rightarrow E_f$ be an automorphism defined by $\sigma(x)=\overline{x}$, for all $x \in E_f$. By Proposition~\ref{GS}, we have $\sigma(C(\mfp,f))={\Psi^* (\mfp)}^{-1}C(\mfp,f)$	 
		for all prime ideals $\mfp$ with $(\mfp,\mfn)=1$. By Proposition~\ref{TRS}, $f$ has a non-trivial inner twist given by  $(\Psi)^{-1}$.
	\end{proof}
 We now give some examples of primitive forms with a non-trivial inner twist.
 \begin{example}
 	\label{example4}
 	Consider a non-CM primitive form $f$ defined over $F=\Q(\sqrt{2})$ of weight $(2, 2)$, level $[41,\; 41,\; 2\sqrt{2}-7]$ and with trivial character. This Hilbert modular form $f$ is labelled as $\texttt{2.2.8.1-41.1-a}$ in \cite{lmfdb}. The coefficient field $E_f$ of $f$ is $\Q(\sqrt{2})$ and $F_f=\Q$.
 \end{example}
\begin{example}
	\label{example5}
	Consider a non-CM primitive form $f$ defined over $F=\Q(\sqrt{3})$ of weight $(2, 2)$, level $[13,\; 13,\; \sqrt{3}+4]$ and with trivial character. This Hilbert modular form $f$ is labelled as $\texttt{2.2.12.1-13.1-a}$ in \cite{lmfdb}. The coefficient field $E_f$ of $f$ is $\Q(\sqrt{2})$ and $F_f=\Q$. 
\end{example}

\begin{example} 
	\label{example6}
	Consider a non-CM primitive form $f$ defined over $F=\Q(\sqrt{6})$ of weight $(2, 2)$, level $[9,\; 3,\; 3]$ and with trivial character. This Hilbert modular form $f$ is labelled as $\texttt{2.2.24.1-9.1-a}$ in \cite{lmfdb}. The coefficient field $E_f$ of $f$ is $\Q(\sqrt{6})$ and $F_f=\Q$.  
\end{example}

        In Example~\ref{example4}, Example~\ref{example5} and Example~\ref{example6}, the coefficient field $E_f \neq F_f$. Hence, these primitive forms $f$ have a non-trivial inner twist.
	\begin{lem}
		\label{no non-trivial inner twists}
		\label{E_f=F_f}
		Suppose $f\in S_k(\mfn, \Psi)$ is a non-CM primitive form over $F$ of weight $k$ and character $\Psi$ such that $[E_f : \Q]$ is an odd prime. If $E_f$ is totally real, then $f$ does not have any non-trivial inner twists. In particular, if $\Psi= \Psi_0$ is a trivial character, then $E_f$ is totally real.
	\end{lem}
	
	\begin{proof}
		Let $\mfp \in P$ be a prime with $(\mfp, \mfn)=1$. Since $E_f$ is totally real, ${C(\mfp, f)}^2 \in F_f$. Since $[E_f : \Q]$ is prime, 
		the field $F_f$ is either $\Q$ or $E_f$. If $F_f =\Q,$ then $[\Q(C(\mfp, f)) : \Q]$ is either $1$ or $2 $. This contradicts to $[E_f : \Q]$ is odd prime. Therefore, $F_f = E_f$. Hence, $f$ does not have any non-trivial inner twists.
	\end{proof}

	\section{Statement and proof of the main theorem} 
	\label{MT1_Section}
    	In this section, we shall state and prove the main theorem of this article.
	\begin{thm}[Main Theorem]
		\label{MT1}
		Let $f\in S_k(\mfn, \Psi)$ be a primitive form defined over $F$ of parallel weight $2$, level $\mfn$,
		and character $\Psi$, which is not a theta series. Further, assume that $[F : \Q]=[E_f : \Q]$ is an odd prime.
		Then $$\delta_D \left( \big \{ \mfp\in P : \Q \big (C(\mfp,f) \big ) =E_f \big \} \right) = 1,$$
		where $\delta_D(S), E_f$ denote the Dirichlet density of $S \subseteq P$, the coefficient field of $f$,
		respectively.
	\end{thm}
	
	\subsection{Images of the residual Galois representations:}
    	We now determine the images of the $\lambda$-adic residual Galois representations attached to a primitive form of parallel weight $2$. The work of Dimitrov in~\cite{D05} is quite influential in this section.
	
    	Let $f\in S_k(\mfn, \Psi)$ be a primitive form defined over $F$ of weight $k=(k_1, k_2, \ldots, k_n)$, level $\mfn$ and character $\Psi$.  Recall that, $k_0=\max \{k_1, \ldots, k_n \}$ and
	     $\omega$ be the mod-$q$ cyclotomic character. Then the function $\bar{\Psi}\omega^{k_0-2}$ is a character on $G_F$. Let $\hat{F}$ be the compositum of the Galois closure of $F$ in $\bar{\Q}$ and the subfield of $\bar{\Q}$ given by $(\bar{F})^{\mathrm{Ker}(\bar{\Psi}\omega^{k_0-2})}$. Then $\hat{F}$ is a Galois extension of $F$ and $G_{\hat{F}} \unlhd G_F$. In~\cite[Proposition 3.9]{D05}, he described the image  $\bar{\rho}_{f, \lambda}(G_{\hat{F}})$ for almost all  $q \in \mathbb{P}$. More precisely, he proved:
	\begin{prop}
	\label{Dthm3}
	    Let $f$ be a primitive form which is not a theta series. Then there exists  a power $\hat{q}$ of $q$ such that for almost all  $q \in \mathbb{P}$, we have either 
	    $$\bar{\rho}_{f, \lambda}(G_{\hat{F}}) \simeq \big\{g \in  \GL_2(\F_{\hat{q}}) : \det(g) \in (\F_q^{\times})^{k_0-1}\big\} ,$$ 
	    or 
	    $$\bar{\rho}_{f, \lambda}(G_{\hat{F}}) \simeq \big \{g \in \F_{\hat{q}^2}^{\times}   \GL_2(\F_{\hat{q}}) : \det(g) \in (\F_q^{\times})^{k_0-1}\big \}.$$
	\end{prop}

    \subsection{Key proposition in the proof of Theorem~\ref{MT1}:}   
    We will now determine the image of $\bar{\rho}_{f, \lambda}$ for primitive forms $f$ as in Theorem~\ref{MT1}.
     More precisely,
		
	\begin{prop}
		\label{image of the primitive forms}
		Let $f\in S_k(\mfn, \Psi)$ be as in Theorem~\ref{MT1}. 
		For any $\lambda \subseteq E_f$ be a prime ideal lying above $q$,
		we have $$\bar{\rho}_{f, \lambda}(G_{F}) \simeq \{\gamma \in \GL_2(\F_{q^{d}}) : \det(\gamma)\in {\F_q^\times} \},$$ 
		for infinitely many $q \in \mathbb{P}$ with $\mathrm{f}(\lambda / q)=d$. 
	\end{prop}
	Before we start the proof of Proposition~\ref{image of the primitive forms}, we recall some necessary results.
    \begin{prop}[\cite{M18}] 
	    \label{Mar1}
	 	Let $K/ \Q$ be a cyclic Galois extension of degree $n$. For $1 \leq r\mid n$, let
	 	$S_r := \{q \in \mathbb{P}: \mathrm{e}(\lambda|q)=1\ \&\ \mathrm{f}(\lambda/q)=r \ \mathit{for\ some\ prime\ ideal}\  \lambda|q\}$. 
	 	Then $\delta_D\left( S_r \right) = \frac{\varphi(r)}{n}$.
	 \end{prop}
     \begin{cor}
     \label{CorollarytoMar1}
      Let $f$ be as in Theorem~\ref{MT1}. 
      Then, there exists infinitely many primes $q \in \mathbb{P}$ which are inert in both $F, E_f$.
     \end{cor}
     For  $q\in \mathbb{P}$, let $\lambda, \upsilon$ be prime ideals of $E_f, F$ lying above $q$, respectively. Let $I_\upsilon$ be the inertia group at $\upsilon$.	 We now give the proof of Proposition~\ref{image of the primitive forms}.

	\begin{proof}
	 	The proof of the proposition is similar to that of the technique in~\cite[Proposition 3.1]{DD06}. 
		In our case, $k_0=2$, $\Psi= \Psi_{\mrm{triv}}$ and hence $G_{\hat{F}}=G_F$. By Proposition~\ref{Dthm3}, for all primes $q \gg 1$,
		there exists a power $\hat{q}$ of $q$, we have either
		$ \bar{\rho}_{f, \lambda}(G_F) \simeq \big\{g \in \GL_2(\F_{\hat{q}}) : \det(g) \in \F_q^{\times}\big\} ,$ or $  \bar{\rho}_{f, \lambda}(G_F) \simeq \big \{g \in \F_{\hat{q}^2}^{\times}   \GL_2(\F_{\hat{q}}) : \det(g) \in \F_q^{\times}\big \}$. We now show that the later possibility will not occur.
		
		Suppose $ \bar{\rho}_{f, \lambda}(G_F) \simeq \{\gamma \in  \F_{\hat{q}^2}^\times\GL_2(\F_{\hat{q}}) : \det(\gamma)\in \F_q^{\times} \}$ for some prime power $\hat{q}$ of $q$ with 
		$q \gg 1$. 
		Now, we argue as in the proof of~\cite[Proposition 3.9]{D05}, we get that $\F_q \subseteq \F_{\hat{q}^2} \subseteq \F_\lambda$. However, this cannot happen because $d$ is  odd  and $2 |[\F_\lambda :\F_q]$. 
		Therefore, 
		$$\bar{\rho}_{f, \lambda}(G_F) \simeq \big\{g \in \GL_2(\F_{\hat{q}}) : \det(g) \in \F_q^{\times} \big\} $$	for $q \in \mathbb{P}$ with $q \gg 1$. By \cite[Proposition 1]{DD06}, the possible fundamental characters for $\bar{\rho}_{f, \lambda}|I_\upsilon$ are of level $d$ or $2d$ if $\mathrm{f}(\upsilon|q)=d$. 
		Hence, we have $\F_{q^d} \subseteq \F_{\hat{q}} \subseteq \F_\lambda$. Now, choose a prime 
		$q \in \mathbb{P}$ which is inert in both $F$ and $E_f$. By Corollary~\ref{CorollarytoMar1}, there
	    exists infinitely many such primes. Since $\mathrm{f}(\lambda|q)=d$, the fundamental characters of level $2d$ cannot occur in $\bar{\rho}_{f, \lambda}|I_\upsilon$, therefore $\F_{q^{d}} = \F_{\hat{q}} =\F_\lambda$. Therefore, we have 
		$$ \bar{\rho}_{f, \lambda}(G_F) \simeq \big\{g \in \GL_2(\F_{q^d}) : \det(g) \in \F_q^{\times} \big\}$$
		for infinitely many $q \in \mathbb{P}$ with $\mathrm{f}(\lambda / q)= \mathrm{f}(\upsilon / q)=d$. 
% 		for all prime $q \gg 1$, which are inert in $E_f$. 
		We are done with the proof.
	\end{proof}

	\subsection{A result for Hilbert modular forms of weight $k$:}
	The aim of this section is to prove an analogue of Theorem~\ref{MT1}
	for Hilbert modular forms of weight $k$. When $k$ is of parallel weight $2$ and $[F : \Q]=[E_f : \Q]$ is an odd prime, we show that the assumption in Theorem~\ref{MT2} is satisfied and hence we use it to prove Theorem~\ref{MT1}. 
	\begin{thm}
    \label{MT2}
    Let $f\in S_k(\mfn, \Psi)$ be a primitive form defined over $F$ of weight $k$, level $\mfn$ and character $\Psi$. For any subfield $\Q \subseteq L \subsetneq E_f$, assume that 
	\begin{equation} 
		\label{assumption in Theorem for E_f}
	  	\bar{\rho}_{f, \lambda}(G_{\hat{F}}) \supseteq 
		\{\gamma \in \GL_2(\F_{q^\mathrm{f}}) : \det(\gamma)\in (\F_q^{\times})^{k_0-1} \}\  \mathrm{with\ } 
		\mathrm{f} = \mathrm{f}(\lambda| q),		 
	\end{equation}
	    for infinitely many $\lambda \in \mathcal{P}_f$ with  $\mathrm{f}_{\lambda, E_f, L} >1$, where $q \in \mathbb{P}$ lying below $\lambda$. Then 
	    $$\delta_D \left( \big \{ \mfp \in P :\, C(\mfp, f)\in L \big \} \right) = 0.$$
	\end{thm} 
	    
	The following proposition of Koo et al., (cf.~\cite[Proposition 2.1(c)]{KSW08}) is helpful in the proof of Theorem~\ref{MT2}.
 
	\begin{prop}
		\label{conjugacy classes}
		Let $R\subseteq \tilde{R}$ be two subgroups of $\F_{q^r}^\times$ for some $q\in \mathbb{P}$ and $r \in \N$. 
		%$\sqrt{\tilde{R}}=\{s\in \F_{q^t}^\times :s^2 \in \tilde{R} \}$ 
		Let $G\subseteq \{g \in \GL_2(\F_{q^r}) : \det(g) \in \tilde{R} \} \leq \GL_2(\F_{q^r})$. Let $P(x)=x^2-ax+b  \in \F_{q^r}[x]$. Then $\sum_{C}|C|\leq  2 |\tilde{R}/R|(q^2+q)$, where the sum carries over all the conjugacy classes $C$ of $G$ with characteristic polynomial equals to $P(x)$.
	\end{prop}
 
	We now start the proof of Theorem~\ref{MT2}. 
	\begin{proof} 
		Let	$\mcO_{E_f}, \mcO_L $ denote the ring of integers of $E_f, L$, respectively. Let $T$ be the set of all prime ideals $\lambda$ of $\mcO_{E_f}$ such that~\eqref{assumption in Theorem for E_f} holds. By assumption, $T$ is an infinite set. For any $Q  \in T$, let $Q_L, q$ be the prime ideals of $\mcO_{L}, \Z$ lying below $Q$, respectively.  
		Let $\F_{q^r}=\mcO_{L}/Q_L,\ \F_{q^{rm}}=\mcO_{E_f}/Q$ for some $r\geq 1, m\geq2$.
		
		Let $R:=(\F_q^{\times})^{k_0-1}$, $W \leq \F_{q^{rm}}^\times$ denote the image of $\Psi^\ast$ mod $Q$ and $\tilde{R}:= \langle R,W \rangle$, the subgroup of $\F_{q^{rm}}^\times$ generated by $R$ and $W$. Then $|R| \leq q-1, \;  |W|\leq |(\mcO_E / \mfn)^\times|$ and $ |\tilde{R}|\leq |R||W|$. 
		
		Let $G:= \Bar{\rho}_{f,Q}(G_F)$ be the image of $Q$-adic residual Galois representation $\Bar{\rho}_{f, Q}$. By \eqref{trace of Galois representation} and \eqref{residual Galois repres}, $G$ is a subgroup 
		of $\{g \in \GL_2(\F_{q^{rm}}) : \text{det}(g) \in\tilde{R} \}$.
		
		Let $M_Q:=\bigsqcup_C \{ \mfp\in P : [\Bar{\rho}_{f, Q}(\Frob_\mfp)]_G=C\}$, where $C$ carries over all the conjugacy classes of $G$ with characteristic polynomial $x^2-ax+b \in \F_{q^{rm}}[x]$ such that $ a \in \F_{q^r}$ and $ b\in \tilde{R}$. Then, there are at most $q^r|R||W|$ such polynomials. By~\eqref{trace of Galois representation}, we have $a \equiv C(\mfp,f) \pmod Q$. Since $ a \in \F_{q^r}$, we get $C(\mfp,f) \pmod Q \in \F_{q^r} $. 
		Hence,
		\begin{equation}
			\label{new Q with the original set} 
			M_Q \supseteq \{\mfp\in P \, :\, C(\mfp, f)\pmod{Q} \in \F_{q^r} \}.
		\end{equation} 
		By Theorem~\ref{Density of conjugacy classes}, we have
		$\delta_D (M_Q)=\sum_{C}\frac{|C|}{|G|}$. Now, by Proposition~\ref{conjugacy classes}, we get 
		\begin{equation}
			\label{new 6.1} 
			\delta_D (M_Q)\leq \frac{q^r|R||W|\times 2 |\tilde{R}/R|(q^2+q)}{|G|} = \frac{2|R||W|^2q^r(q^2+q)}{|G|}.
		\end{equation}
        Since $G_F \supseteq G_{\hat{F}}$, by~\eqref{assumption in Theorem for E_f}, we get 
		\begin{equation} \label{new 6.2}
			|G|\geq |R|\times  |\SL_2(\F_{q^{rm}})|,
		\end{equation}
		and this provides a lower bound to $|G|$. Combining~\eqref{new 6.2} with~\eqref{new 6.1}, we get 
		$$ \delta_D (M_Q)\leq \frac{2|W|^2|R| q^r(q^2+q)}{|R|\times  |\SL_2(\F_{q^{rm}})|}= \frac{2|W|^2{q^{r+3}}}{q^{3rm}(q-1)}.$$ Since $m\geq 2, r\geq 1$, we have $\delta_D (M_Q)\leq O\left(\frac{1}{q^2}\right)$. By hypothesis, $T$ is an infinite set and hence $q$ is unbounded.
		The inclusion of the sets in~\eqref{new Q with the original set} implies
		$\big \{ \mfp \in P :\, C (\mfp, f)\in L \big \} \subseteq \bigcap_{Q \in T} M_Q
		$. Therefore, we have
		
		\begin{equation}
		\label{conclusion of Theorem 4.1}
			\delta_D \left( \big \{ \mfp \in P :\, C(\mfp, f)\in L \big \} \right) = 0.
		\end{equation}
		This completes the proof of an auxiliary result, i.e., Theorem~\ref{MT2}.
\end{proof}
The above theorem holds even if the inclusion in~\eqref{assumption in Theorem for E_f} holds up to conjugation.
\begin{cor}
\label{Corollary to MT2}
        Let $f$ be as in Theorem~\ref{MT2} which satisfies \eqref{assumption in Theorem for E_f}, for any  subfield $\Q \subseteq L \subsetneq E_f$. Then
        $\delta_D \left( \big \{ \mfp\in P : \Q \big (C(\mfp,f) \big ) =E_f \big \} \right) = 1.$
\end{cor}
\begin{proof}
         Let $\mfp\in P$ be a prime with $ \Q \big (C(\mfp,f) \big ) \subsetneq E_f $.Then $C(\mfp,f) \in L$ for some proper subfield $L$ of $E_f$.
         Since $[E_f:\Q]$ is a finite separable extension, there are only finitely many subfields between $\Q$ and $E_f$, and by~Theorem~\ref{MT2}, we have that $\delta_D \left( \big \{ \mfp\in P : \Q \big (C(\mfp,f) \big ) \subsetneq E_f \big \} \right) =0.$ This completes the proof of the corollary.
\end {proof}

We have some remarks to make.
 \begin{itemize}
  \item The conclusion of Theorem~\ref{MT2} implies that $f$ has to be a non-CM form; For a CM form, the density of $\mfp \in P$ for which $C(\mfp,f)=0$ has density $1/2$.		
  \item The equation~\eqref{assumption in Theorem for E_f} of Theorem~\ref{MT2} implies that $E_f \neq \Q$. However, if $E_f = \Q$ then $C(\mfp, f) \in \Q$ for all $\mfp \in P$. The density of $\{ \mfp \in P :\, \Q (C(\mfp, f))=\Q \big \}$ is $1$ if $f$ is non-CM, is $1/2$ if $f$ is CM. 
 \end{itemize}

\subsection{The proof of Theorem~\ref{MT1} with supporting examples}
        In this section, we give a proof of Theorem ~\ref{MT1} and provide some examples of $f$ in support of it.
\begin{proof}[Proof of Theorem~\ref{MT1}]
	    Since $[E_f : \Q]$ is an odd prime, there only proper subfield $L$ of $E_f$ is $\Q$. By Proposition~\ref{image of the primitive forms}, $f$ satisfies the assumption~\eqref{assumption in Theorem for E_f} of Theorem~\ref{MT2}. Hence, by Corollary~\ref{Corollary to MT2}, the proof of Theorem~\ref{MT1} follows. 
\end{proof}
   We now give some examples of primitive Hilbert modular forms $f$
   in support of Theorem~\ref{MT1}.

\begin{example}
\label{example1}
        Consider a non-CM primitive form $f$ defined over $F=\Q(\zeta_7)^+$ with generator $\omega$ having minimal polynomial $x^3-x^2-2x+1$, with weight $(2, 2,2)$, level $[167,\; 167,\; \omega^2+\omega-8]$ and with trivial character. This Hilbert modular form $f$ is labelled as $\texttt{3.3.49.1-167.1-a}$ in \cite{lmfdb}. The coefficient field $E_f$ of $f$ is $\Q(\alpha)$, where $\alpha$ is a root of the irreducible polynomial $x^3-x^2-4x-1 \in \Q[x]$.
%         The field $E_f$ is a totally real and $[F : \Q]=[E_f : \Q]=3$. 
        
\end{example}
	
	\begin{example}
		\label{example2}
		Consider a non-CM primitive form $f$ defined over  $F=\Q(\zeta_9)^+$ with generator $\omega$ having minimal polynomial $x^3-3x-1$, with weight $(2, 2,2)$, level $[71,\; 71,\; \omega^2+\omega-7]$ and with trivial character. This Hilbert modular form $f$ is labelled as  $\texttt{3.3.81.1-71.1-a}$ in \cite{lmfdb}. The coefficient field $E_f$ of $f$ is $\Q(\beta)$, where $\beta$ is a root of the irreducible polynomial $x^3-x^2-4x+3 \in \Q[x]$. 
%		The field $E_f$ is a totally real and $[F : \Q]=[E_f : \Q]=3$. 

	\end{example}
	
	\begin{example}
		\label{example3}
		Consider a non-CM primitive form $f$ defined over  $F=\Q(\zeta_7)^+$ with generator $\omega$ having minimal polynomial $x^3-x^2-2x+1$, with weight $(2, 2,2)$, level $[239,\; 239,\; 6\omega^2-5\omega-7]$ and with trivial character. This Hilbert modular form $f$ is labelled as $\texttt{3.3.49.1-239.1-a}$ in \cite{lmfdb}. The coefficient field $E_f$ of $f$ is $\Q(\theta)$, where $\theta$ is a root of the irreducible polynomial $x^3-12x-8 \in \Q[x]$.
%		 The field $E_f$ is a totally real and $[F : \Q]=[E_f : \Q]=3$. 
	\end{example}

     The  primitive modular forms $f$ in Example~\ref{example1}, Example~\ref{example2}, and Example~\ref{example3} are of parallel weight $2$ with $[F : \Q]=[E_f:\Q]=3$ and hence they satisfies the hypothesis of Theorem~\ref{MT1}. Moreover, $E_f$ is totally real and hence by Lemma~\ref{E_f=F_f}, these primitive forms $f$ do not have any non-trivial inner twists.

     \subsection{Computation of some Dirichlet density for $F_f$}
		In this section, we shall state and prove a variant of Theorem~\ref{MT2} and Corollary~\ref{Corollary to MT2} for $F_f$. 
		In fact, we compute the Dirichlet density of the set $\big \{ \mfp \in P :\, \Q\left(C^\ast(\mfp, f)\right)=F_f \big \}.$
		
		\begin{thm}
			\label{MT3}
			Let $f\in S_k(\mfn, \Psi)$ be a primitive form defined over $F$ of weight $k$, level $\mfn$ and character $\Psi$. For any  subfield $\Q \subseteq L \subsetneq F_f$, assume that 
			\begin{equation} 
				\label{assumption in Theorem for F_f}
				\bar{\rho}_{f, \lambda}(G_{\hat{F}})  \supseteq 
				\{\gamma \in \GL_2(\F_{q^{\mathrm{f}}}) : \det(\gamma)\in (\F_q^{\times})^{k_0-1} \}\  \mathrm{with\ } 
				\mathrm{f} = \mathrm{f}_{\lambda, F_f, \Q},
			\end{equation}
			for infinitely many $\lambda \in \mathcal{P}_f$ with $\mathrm{f}_{\lambda, F_f, L} >1$, where 
			$q \in \mathbb{P}$ is the rational prime lying below $\lambda$. Then
			$$\delta_D \left( \big \{ \mfp \in P :\, C^\ast(\mfp, f)\in L \big \} \right) = 0.$$
	     	\end{thm}
     	The above theorem holds even if the inclusion in~\eqref{assumption in Theorem for F_f} holds up to conjugation.

		    \begin{proof} 
			In this proof, we follow the notations as in Theorem~\ref{MT2}.
			Let	$\mcO_{F_f}$ be the ring of integers of $F_f$.
			For any $Q \in T$, let $Q_F$ be the prime ideal of $\mcO_{F_f}$ lying below $Q$.
			Let $\mcO_{L}/Q_L=\F_{q^{r}}, \mcO_{F_f}/Q_F= \F_{q^{rm}}$ and $ \mcO_{E_f}/ Q= \F_{q^{rms}}$, for some $r\geq 1,\ m\geq 2$ and $s\geq 1$. 
			Then $G\subseteq \{g \in \GL_2(\F_{q^{rms}}) : \det(g) \in\tilde{R} \}$. Now, arguing as in the proof of Theorem~\ref{MT2}, we get
			$\delta_D (M_Q)\leq 			\frac{4|W|^3{q^{r+3}}}{q^{3rm}(q-1)}.$
			Since $m\geq 2, r\geq 1$, we get $\delta_D (M_Q)\leq O\left(\frac{1}{q^2}\right)$.
			Therefore, we have
			\begin{equation*} 
% 				\label{conclusion of MT3}
			\delta_D \left(\{ \mfp \in P :\, C^\ast(\mfp, f)\in L \big \}\right)=0.
		\end{equation*}
	This completes the proof of  Theorem~\ref{MT3}. 
\end{proof}

% We have now some remarks to make.
%    	\begin{itemize}
%    		
%    		\item The conclusion of  Theorem~\ref{MT3} implies that $f$ has to be non-CM; 
%    		For a CM form, the density of $\mfp \in P$ for which $C^\ast(\mfp,f)=0$ has density $1/2$.
%    		
%    		\item The assumption~\eqref{assumption in Theorem for F_f} in Theorem~\ref{MT3} implies that $F_f \neq \Q$. 
%    		However, if $F_f = \Q$ then $C^\ast(\mfp, f) \in \Q$ for all $\mfp \in P$. The density of $\{ \mfp \in P :\, \Q (C^\ast(\mfp, f))=\Q \big \}$ is $1$ if $f$ is non-CM, is $1/2$ otherwise. 
%    		
%    	\end{itemize}
   
		\begin{cor}
			\label{Corollary1 to MT3}
			Let $f$ be as in Theorem~\ref{MT3} which satisfies \eqref{assumption in Theorem for F_f}, for any  subfield $\Q \subseteq L \subsetneq F_f$. Then
			$$\delta_D \left( \big \{ \mfp\in P : \Q \big (C^\ast(\mfp,f) \big ) =F_f \big \} \right) = 1.$$
		\end{cor}
	 
		\begin{proof}
			Suppose $\mfp\in P$ is a prime such that $L=\Q \big (C^\ast(\mfp,f) \big ) \subsetneq F_f $ is a proper subfield of $F_f$. Since $[F_f:\Q]$ is a finite separable extension, there are only finitely many subfields between $\Q$ and $F_f$, and by Theorem~\ref{MT3}, we get
			$$\delta_D \left( \big \{ \mfp\in P : \Q \big (C^\ast(\mfp,f) \big ) \subsetneq F_f \big \} \right) =0.$$ 
			This completes the proof of the corollary.	
		\end{proof}
	
			\begin{cor}\label{Corollary2 to MT3}
				Let $f$ and $\bar{\rho}_{f, \lambda}$ be as in Theorem~\ref{MT3}. Then, we have
				$$ \delta_D \left( \big \{ \mfp\in P : F_f \subseteq \Q(C(\mfp,f)) \big \} \right) = 1.$$
			\end{cor}
		
			\begin{proof}
				Suppose $\mfp \in P$ with $\Q(C^\ast(\mfp,f) )=F_f$. From Lemma ~\ref{main lemma}, we have $F_f=\Q(C^\ast(\mfp,f) ) \subseteq \Q(C(\mfp,f))$. Corollary~\ref{Corollary1 to MT3} implies the result.
			\end{proof}
		
		In Example~\ref{example1}, Example~\ref{example2}, and Example~\ref{example3}, we have $E_f$ and $F$ are of degree $3$ over $\Q$ and $E_f=F_f$. Since there are no proper subfields of $F_f$, by Proposition~\ref{image of the primitive forms}, we conclude that these examples satisfy the hypothesis \eqref{assumption in Theorem for F_f} of Theorem~\ref{MT3}.

	    \section{Computation of the Dirichlet density for subfields of $E_f$}
			\label{final section}
         In $\S\ref{MT1_Section}$, we have computed the Dirichlet density of $\mfp \in P$ such that $C(\mfp, f)$ (resp., $C^\ast(\mfp, f)$) generates $E_f$ (resp., $F_f$). In this section, for any subfield $K$ of $E_f$, we compute the Dirichlet density of the set $\big \{\mfp \in P :  \Q \big(C(\mfp,f)\big)= K \big \}$. It is quite surprising to see that this density depends on whether $K\supseteq F_f$ or not.
			
			We now calculate the density of $\mfp\in P$ such that 
			$C(\mfp,f)\in K$ when $F_f \nsubseteq K$. 
			The following lemma is an analog of~\cite[Corollary 1.3(a)]{KSW08} for classical modular forms. 
			
			\begin{lem}\label{K does not contains F_f} 
				Let $f$ be as in Theorem \ref{MT3}. Let $K \subseteq E_f$ be a subfield such that $F_f \nsubseteq K$. Then $ \delta_D \left( \big \{\mfp \in P :  C(\mfp,f) \in K \big \}\right)= \delta_D \left( \big \{\mfp \in P :  \Q \big(C(\mfp,f)\big)= K \big \}\right) =0.$
			\end{lem}
			
			\begin{proof}
				Since $F_f \nsubseteq K$, we get 
				$ \big \{\mfp \in P :  C(\mfp,f) \in K \big \} \subseteq \big \{ \mfp\in P : F_f \nsubseteq \Q(C(\mfp, f))\big \}$. 
				The proof now follows from Corollary~\ref{Corollary2 to MT3}.
			\end{proof}
			
			Let $\Gamma^\prime=\{\gamma_1,\dots,\gamma_r\} \leq \Gamma$ be a subgroup associated with all the inner twists of $f$ and let $\Psi_{\gamma_1},\dots, \Psi_{\gamma_r}$ be the corresponding Hecke characters. 
			Hence, the ideal Hecke characters $\Psi_{\gamma_1}^\ast,\dots, \Psi_{\gamma_r}^\ast$ can be thought of as characters on $G_F$. For each $i\in \{1, 2, \ldots, r\}$,
			define $H_{\gamma_i} :=\mrm{Ker}(\Psi_{\gamma_i}^\ast)$ and set $H^{\Gamma^\prime} :=\bigcap_{i=1}^{r} H_{\gamma_i} \leq G_F$. Let $K_{H^{\Gamma^\prime}}$ denote the fixed field of $H^{\Gamma^\prime}.$
			Hence, $F\subseteq K_{H^{\Gamma^\prime}} \subseteq \Bar{F}$.
			
			\begin{lem}
				\label{density of splits completely prime}
				Let $\Gamma^\prime, H^{\Gamma^\prime}$ and $ K_{H^{\Gamma^\prime}}$ be as above. Then $$ \{ \mfp \in P : \mfp \  \textrm{splits completely in}\ K_{H^{\Gamma^\prime}} \}= \{ \mfp \in P :  \Psi_\gamma^\ast(\mfp)=1,\ \forall \gamma \in \Gamma^\prime \big \}.$$
			\end{lem}
			\begin{proof}
				Let $m= [K_{H^{\Gamma^\prime}}: F]$.
				\begin{align*}
					& \{ \mfp \in P : \ \mfp\ \text{splits completely in}\  K_{H^{\Gamma^\prime}}\}\\ 
					& = \{ \mfp \in P :\ \mfp= \mfp_1 \mfp_2 \dots \mfp_m  \text{ for prime ideals } \mfp_j \ \text{in} \ K_{H^{\Gamma^\prime}} \ \text{with} \ 1\leq j \leq m\} \\
					&=\{ \mfp \in P : \ \Psi_{\gamma_i}^\ast(\mfp_j)=1\ \forall i \in \{1,2, \dots, r\},\ \forall j\in \{1,2, \dots, m\} \}\\
					& =\{ \mfp \in P : \ \Psi_{\gamma_i}^\ast(\mfp)=1\ \forall i\in \{1,2, \dots, r\}\}.
				\end{align*}
				This completes the proof of the lemma.
			\end{proof}
			We are now in a position to compute the density of the set
			$\big \{\mfp \in P:  C(\mfp,f) \in K  \big \}$ if $F_f \subseteq K$. The following proposition generalizes~\cite[Corollary 1.3(b)]{KSW08} to primitive Hilbert modular forms. 
			
			\begin{prop} \label{K contains F_f}
				Let $f\in S_k(\mfn, \Psi)$ be a non-CM primitive form  defined over $F$ of weight $k$ and character $\Psi$. For any subfield $K$ with $F_f \subseteq K \subseteq E_f$, there exists a subgroup $\Gamma'$ of $\Gamma$ such that $K=E_f^{\Gamma'}$ and 
				$ \delta_D \left( \big \{ \mfp \in P   :  C(\mfp,f) \in K  \big \}\right) =\frac{1}{[K_{H^{\Gamma'}}: F]}$.
			\end{prop}
			
			\begin{proof}
				Since $E_f/F_f$ is Galois, there exists $\Gamma' \subseteq \Gamma$ such that $K=E_f^{\Gamma'}$.
				Hence,
				\begin{equation*}
					\begin{split}
						\{\mfp \in P :  C(\mfp,f) \in K \}  
						&=\{\mfp \in P :  \gamma(C(\mfp,f))= C(\mfp,f) \text{ for all } \gamma \in \Gamma^\prime  \} \\
						&=\{\mfp \in P : \Psi_\gamma^\ast(\mfp) C(\mfp,f)= C(\mfp,f)  \text{ for all } \gamma \in \Gamma^\prime  \}. 
					\end{split}
				\end{equation*}
				Since $\delta_D \left( \big \{\mfp \in P : C(\mfp,f)=0 \big \}\right)=0$ (cf. \cite[Theorem 4.4(1)]{DK}) and by
				Chebotarev density theorem, we have
				\begin{equation*}
					\begin{split}
						&\delta_D \left( \big \{\mfp \in P :   \Psi_\gamma^\ast(\mfp) C(\mfp,f)= C(\mfp,f)\   \textrm{for all}\  \gamma \in \Gamma' \big \} \right)
						\\
						&\qquad \qquad = \delta_D \left( \big \{ \mfp \in P :  \Psi_\gamma^\ast(\mfp)=1\   \textrm{for all}\  \gamma \in \Gamma' \big \} \right) \\
						&\qquad \qquad \underset{\mrm{Lemma}~\ref{density of splits completely prime}}{=} \delta_D \left( \big\{\mfp \in P : \mfp \text{ splits completely in } K_{H^{\Gamma'}} \big  \} \right) =\frac{1}{[K_{H^{\Gamma'}}: F]}. 
					\end{split}
				\end{equation*}
				This completes the proof.
			\end{proof}
		
			The following corollary is an application of Proposition~\ref{K contains F_f}
			and an analog of~\cite[Corollary 1.4]{KSW08} for classical modular forms.
			  
			\begin{cor}
				\label{Corollary to K contains F_f}
				Let $f$, $K$ be as in Proposition~\ref{K contains F_f} and  $K=E_f^{\Gamma'}$ for  $\Gamma' \leq  \Gamma$. Then 
				\begin{align*}
					&\delta_D \left( \big \{ \mfp \in P : \Q \big( C(\mfp,f) \big) =K \big \} \right) \\
					&\qquad = \delta_D \left( \big \{ \mfp \in P : \Psi_\gamma^\ast(\mfp)=1, \; \forall \gamma \in \Gamma' \ \mathrm{and}\      \Psi_\omega^\ast(\mfp)\neq 1, \; \forall \omega \in \Gamma - \Gamma' \big \} \right).
				\end{align*}
			\end{cor}
			These results illustrate that the Dirichlet density of  $\mfp \in P$ such that $\Q(C(\mfp, f)) = K$, with $F_f \subseteq K \subseteq E_f$, is determined by the inner twists of $f$ associated with $K$.
			
\section*{Acknowledgments}  
The authors are thankful to the editor for the constant support and the anonymous referee for the valuable suggestions towards the improvement of this paper.


\begin{thebibliography}{abcd9999}
				
				
				\bibitem[Car86]{C86}
				Carayol, Henri.
				Sur les repr\'esentations $l$-adiques associ\'ees aux formes modulaires de Hilbert. (French) [On $l$-adic representations associated with Hilbert modular forms]
				Ann. Sci. École Norm. Sup. (4) 19 (1986), no. 3, 409–468.
				
%				\bibitem[CS01]{CS}
%			    Consani, Caterina; Scholten, Jasper.
%			    Arithmetic on a quintic threefold.
%			    Internat. J. Math. 12 (2001), no. 8, 943–972.
				
				\bibitem[DK20]{DK}
				Dalal, Tarun; Kumar, Narasimha.
				On non-vanishing and sign changes of the Fourier coefficients of Hilbert cusp forms. 
				Topics in Number Theory, Ramanujan Math. Soc. Lect. Notes Ser., (2020), no. 26, 175--188. 			
				
				
%				\bibitem[Dem02]{D02}
%				Demb\'el\'e, Lassina.
%				Explicit computations of Hilbert modular forms on 
%				$\Q(\sqrt{5})$. Thesis (Ph.D.)-McGill University (Canada). ProQuest LLC, Ann Arbor, MI, 2002. 139 pp.   
				
				\bibitem[DD06]{DD06} 
				Dieulefait, Luis; Dimitrov, Mladen.
				Explicit determination of images of Galois representations attached to Hilbert modular forms. 
				J. Number Theory 117 (2006), no. 2, 397--405.
				
				\bibitem[Dim05]{D05} 
				Dimitrov, Mladen.
				Galois representations modulo $p$ and cohomology of Hilbert modular varieties. 
				Ann. Sci. École Norm. Sup. (4) 38 (2005), no. 4, 505--551.      
				
				\bibitem[KSW08]{KSW08}
				Koo, Koopa Tak-Lun; Stein, William; Wiese, Gabor.
				On the generation of the coefficient field of a newform by a single Hecke eigenvalue.
				J. Theor. Nombres Bordeaux 20 (2008), no. 2, 373--384.
				
				\bibitem[LMFDB]{lmfdb}
				The LMFDB Collaboration.
				The $L$-functions and Modular Forms Database,
				\mbox{\url{http://www.lmfdb.org}} (2021) [Online; accessed \today].
				
				\bibitem[Mar77]{M18}
				Marcus, Daniel A. 
				Number fields. 
				Universitext Springer-Verlag, New York-Heidelberg, 1977.
				
%				\bibitem[Oka02]{O02}
%				Okada, Kaoru.
%				Hecke eigenvalues for real quadratic fields.
%				Experiment. Math. 11 (2002), no. 3, 407--426. 
				
% 				\bibitem[Rib80]{R80}
% 				Ribet, Kenneth A.
% 				Twists of modular forms and endomorphisms of abelian varieties.
% 				Math. Ann. 253 (1980), no. 1, 43--62.
				
				\bibitem[Rib85]{R85} 
				Ribet, Kenneth A. 
				On $l$-adic representations attached to modular forms.II.
				Glasgow Math. J. 27 (1985), 185--194.
				
				\bibitem[Ser81]{S81}
				Serre, Jean-Pierre.
				Quelques applications du th\'eor\`eme de densit\'e de Chebotarev. (French) 
				[[Some applications of the Chebotarev density theorem]]
				Inst. Hautes Études Sci. Publ. Math. No. 54 (1981), 323--401.
				
				\bibitem[SW93]{SW93} 
				Shemanske, Thomas R.; Walling, Lynne H.
				Twists of Hilbert modular forms.
				Trans. Amer. Math. Soc. 338 (1993), no. 1, 375--403. 
				
				\bibitem[Shi78]{S78}
				Shimura, Goro.
				The special values of the zeta functions associated with Hilbert modular forms. 
				Duke Math. J. 45 (1978), no. 3, 637--679.
				
				\bibitem[Tay89]{T89}
				Taylor, Richard.
				On Galois representations associated to Hilbert modular forms.
				Invent. Math. 98 (1989), no. 2, 265--280.	
				
			\end{thebibliography}
		\end{document}